\documentclass[a4wide,11pt]{amsart}
\usepackage[utf8]{inputenc}
\usepackage[english]{babel}
\usepackage[T1]{fontenc}
\usepackage{amsfonts}
\usepackage{geometry}
\usepackage{amsmath, amsthm, amssymb}
\usepackage{tikz}
\usepackage[bookmarks=true]{hyperref}
\theoremstyle{plain}
\newtheorem{theorem}{Theorem}[section]
\newtheorem{proposition}[theorem]{Proposition}
\newtheorem{corollary}[theorem]{Corollary}
\newtheorem{lemma}[theorem]{Lemma}
\theoremstyle{definition}
\newtheorem{definition}[theorem]{Definition}
\newtheorem{example}[theorem]{Example}

\newtheorem{construction}[theorem]{Construction}
\newtheorem{remark}[theorem]{Remark}
\newtheorem{question}[theorem]{Question}

\newcommand{\rank}{{\rm rank\ }}
\def\NN{\mathbb{N}}
\def\ZZ{\mathbb{Z}}
\date{}

\title{The poset of proper divisibility}

\author[D. Bolognini, A. Macchia, E. Ventura, V. Welker]{Davide Bolognini, Antonio Macchia, Emanuele Ventura, Volkmar Welker}

\address[D. Bolognini]{\small Fachbereich Mathematik und Informatik, Philipps-Universit\"at Marburg, Hans-\break Meerwein-Strasse 6, 35032 Marburg, Germany}
\email{\small davide.bolognini@yahoo.it}

\address[A. Macchia]{\small Fachbereich Mathematik und Informatik, Philipps-Universit\"at Marburg, Hans-\break Meerwein-Strasse 6, 35032 Marburg, Germany}
\email{\small macchia.antonello@gmail.com}

\address[E. Ventura]{\small Department of Mathematics and Systems Analysis, Aalto University, Helsinki, Finland}
\email{\small emanuele.ventura@aalto.fi}

\address[V. Welker]{\small Fachbereich Mathematik und Informatik, Philipps-Universit\"at Marburg, Hans-\break Meerwein-Strasse 6, 35032 Marburg, Germany}
\email{\small welker@mathematik.uni-marburg.de}

\begin{document}
\begin{abstract}
  We study the partially ordered set $P(a_1,\ldots, a_n)$ of
  all multidegrees $(b_1,\dots,b_n)$ of monomials
  $x_1^{b_1}\cdots x_n^{b_n}$ which properly divide $x_1^{a_1}\cdots x_n^{a_n}$.
  We prove that the order complex $\Delta(P(a_1,\dots,a_n))$ of
  $P(a_1,\ldots a_n)$ is (non-pure) shellable,
  by showing that the order dual of $P(a_1,\ldots,a_n)$ is
  $\mathrm{CL}$-shellable. Along the way, we exhibit the poset $P(4,4)$ as
  a new example of a poset with $\mathrm{CL}$-shellable order dual that is not
  $\mathrm{CL}$-shellable itself.
  For $n = 2$ we provide the rank of all homology groups of the order
  complex $\Delta \left( P(a_1,a_2) \right)$. Furthermore, we give a
  succinct formula for the Euler characteristic of
  $\Delta \left(  P(a_1,a_2) \right)$.
\end{abstract}

\thanks{The first author was supported by Università degli Studi di Genova}
\thanks{The second author was supported by a DAAD Research scholarship for Post-docs}

\maketitle

\section{Introduction}

In this paper we study proper divisibility of monomials in the polynomial ring in $n$ variables $x_1,\ldots, x_n$. Since any monomial $x_1^{a_1} \cdots x_n^{a_n}$ is determined by its exponent vector $(a_1,\ldots, a_n) \in \NN^n$, we phrase all concepts in terms of exponent vectors.

For every $(a_1,\ldots, a_n), (b_1,\ldots, b_n) \in \NN^n$, we say that $(a_1,\ldots, a_n)$ \textit{properly divides} $(b_1,\ldots, b_n)$ if for every $1 \leq i \leq n$, either $a_i = b_i = 0$ or $a_i < b_i$.
Proper divisibility of monomials appears naturally in the context of the Buchberger algorithm from Gr\"obner basis theory and it plays an important role in the combinatorics of free resolutions of monomial ideals, as shown by Miller and Sturmfels \cite{MS99,MS05} (see also \cite{OW14}).

Here we consider proper divisibility as an order relation, setting $(a_1,\ldots, a_n) \leq (b_1,\ldots, b_n)$ if and only if either $(a_1,\ldots, a_n) = (b_1,\ldots, b_n)$ or $(a_1,\ldots, a_n)$ properly divides $(b_1,\ldots, b_n)$. For an arbitrary $(a_1,\ldots, a_n) \in \NN^n$, we set
$$
P(a_1,\ldots, a_n) = \{ (b_1,\ldots, b_n) \in \NN^n : (b_1,\ldots, b_n) \leq (a_1,\ldots, a_n) \}
$$
and consider $P(a_1,\ldots, a_n)$ as a partially ordered set, poset for short, ordered by proper divisibility. With this order, $P(a_1,\ldots, a_n)$ has a unique minimal element $\widehat{0} = (0,\ldots,0)$ and a unique maximal element $\widehat{1} = (a_1,\ldots,a_n)$. The poset of all divisors of $(a_1,\ldots,a_n)$ with respect to the usual divisibility relation is well understood and it is a direct product of $n$ chains of length $a_1,\ldots,a_n$ respectively. Our approach to study $P(a_1,\ldots,a_n)$ is topological. We associate to $P(a_1,\ldots,a_n)$ its order complex $\Delta(P(a_1,\ldots,a_n))$, which is the simplicial complex consisting of all chains in $P(a_1,\ldots,a_n) \setminus \{ \widehat{0},\widehat{1} \}$. Through order complex and geometric realization, we can study a poset in topological terms and, in particular, we can talk about homotopy equivalence and homology of posets. Note that, in general, not all maximal chains in $P(a_1,\ldots, a_n)$ have equal length and hence its order complex is not pure. A central tool for describing the topology of posets is (non-pure) shellability (see \cite{BW96,BW97}). In our first result we prove shellability for $\Delta ( P(a_1,\ldots, a_n) )$.

\begin{theorem} \label{T.shellable}
For all $(a_1,\ldots, a_n) \in \NN^n$, the order complex $\Delta(P(a_1,\ldots, a_n))$ is shellable. In particular, $\Delta \left( P(a_1,\ldots, a_n) \right)$ is homotopy equivalent to a wedge of spheres and its reduced simplicial homology groups
$\widetilde{H}_\bullet \left( \Delta \left( P(a_1,\ldots, a_n) \right);\ZZ \right)$ are torsion-free.
\end{theorem}

Indeed, we prove a slightly stronger statement. In Theorem \ref{T.dualCLshellable} we show that the poset $P(a_1,\ldots, a_n)^*$, which is the order dual of $P(a_1,\ldots,a_n)$, admits a recursive atom ordering. By \cite[Thm. 5.11]{BW96}, this is equivalent to $\mathrm{CL}$-shellability, which in turn, by \cite[Thm. 11.6]{BW97} and \cite[Thm. 5.8]{BW96}, implies vertex-decomposability, and hence shellability, of the order complex $\Delta(P(a_1,\ldots, a_n)^*)$. Then Theorem \ref{T.shellable} follows by the isomorphism $\Delta(P(a_1,\ldots, a_n)^*) \cong \Delta(P(a_1,\ldots, a_n))$.

$\mathrm{CL}$-shellability and recursive atom orderings are concepts defined for non-pure posets in \cite{BW96,BW97}. The posets $P(a_1,\ldots,a_n)$ provide a good source of counterexamples in this context.

\begin{proposition} \label{P.notshellable}
  The poset $P(4,4)$ is not $\mathrm{CL}$-shellable but its dual $P(4,4)^*$ is.
  In particular, $\Delta(P(4,4))$ is shellable but $P(4,4)$ is not $\mathrm{CL}$-shellable.
\end{proposition}

In \cite{W85} Walker provides an example of a pure poset whose order complex is shellable, but which is not $\mathrm{CL}$-shellable. Our example, $P(4,4)$, is not pure but smaller both in the number of elements and in dimension than the example from \cite{W85}. The question if there is a non-$\mathrm{CL}$-shellable poset whose dual is $\mathrm{CL}$-shellable was posed as an open question by Wachs in \cite[p.\,71]{W07}. Already in 2008 it was answered by Schweig \cite{SW08}, who provided a counterexample of the same dimension as ours and of almost equal size, but which is
even pure. Note that, the example by Schweig also provides a poset that is shellable but not $\mathrm{CL}$-shellable.

Using the $\mathrm{CL}$-shelling from Theorem \ref{T.dualCLshellable} one can in principle read off the homotopy type and the homology groups. Since the process is technically involved, we present a pleasing solution for the case $n=2$ only. To simplify the notation, in this case we set $a=a_1$ and $b=a_2$ and, without loss of generality, we may assume $a \leq b$.

\begin{theorem} \label{T.homology}
  Let $2\leq a\leq b$.
  Then
  $\widetilde{H}_i  (\Delta ( P(a,b) ) ; \ZZ ) = 0$ for $i > a-2$ and
  \begin{equation}\label{Eq.homology}
    \mathrm{rank}\ \widetilde H_i (\Delta ( P(a,b) ); \mathbb Z )=
    2 \sum_{t=0}^i \binom{a-3-i}{t-1} \bigg[ \binom{i}{t}\binom{b-2-i}{i-t} + \binom{i}{t-1}\binom{b-3-i}{i-t} \bigg],
  \end{equation}
  for $0\leq i\leq a-2$, where we set $\binom{-1}{-1}=1$.
\end{theorem}

A remarkable property of these posets is the following persistence theorem,
which is a phenomenon rarely observed in posets defined {\it naturally} on
combinatorial objects.

\begin{proposition} \label{P.nojumps}
  For every $a,b$, with $2 \leq a \leq b$, there exists an integer $t_{(a,b)} \geq 0$ such that
  \[
     H_i ( \Delta (P(a,b)); \mathbb Z ) \neq 0 \text{ if and only if } 0 \leq i \leq t_{(a,b)},
  \]
  where $H_i$ denotes the $i$\textsuperscript{th} (non-reduced) simplicial homology group.
\end{proposition}

Moreover, in Corollary \ref{C.contractible} we show that the only poset $P(a,b)$, with $2 \leq a \leq b$, whose order complex is contractible, and indeed collapsible, is $P(3,3)$.

One of the most important numerical invariants of a poset $P$ with unique minimal element $\widehat{0}$ and unique maximal element $\widehat{1}$ is its M\"obius number $\mu(P)$ \cite{R64}. It is well known that $\mu(P)$ is the alternating sum of the ranks of the homology groups of the order complex $\Delta(P)$. In particular, it equals the reduced Euler characteristic $\widetilde \chi ( \Delta(P) )$ of $\Delta(P)$. Hence, as a consequence of Theorem \ref{T.homology}, we can derive the following formula for the reduced Euler characteristic of $\Delta ( P(a,b) )$.

\begin{theorem} \label{T.EulerChar}
For every $2 \leq a \leq b$, the reduced Euler characteristic of $\Delta ( P(a,b) )$ is
\begin{equation} \label{Eq.EulerChar}
\widetilde \chi ( \Delta ( P(a,b) ) ) = (-1)^a \cdot 2 \sum_{h=0}^{\lfloor \frac{a}{2} \rfloor -1} (-1)^h \binom{a-2}{h} \binom{b-a}{a-2-2h}.
\end{equation}
\end{theorem}

The posets $P(a_1,\ldots, a_n)$ can be seen as examples of the following general construction. Let $P_1,\ldots,P_n$ be posets. Assume that, for $1 \leq i \leq n$, the poset $P_i$ has unique minimal element $\widehat{0}_i$ and unique maximal element $\widehat{1}_i$. For every two elements $(a_1,\ldots,a_n), (b_1,\ldots, b_n)$ in the Cartesian product $P_1 \times \cdots \times P_n$ we set $(a_1,\ldots,a_n) \leq_p (b_1,\ldots, b_n)$ if
$(a_1,\ldots,a_n) = (b_1,\ldots, b_n)$ or, for every $i$, either $a_i = b_i = \widehat{0}_i$ or $a_i < b_i$ in $P_i$. We write $P_1\times_p \cdots \times_p P_n$ for the set of all $(a_1,\ldots, a_n) \in P_1 \times \cdots \times P_n$ with $(a_1,\ldots, a_n) \leq_p (\widehat{1}_1,\ldots, \widehat{1}_n)$. If we denote by $C_{k+1}$ a chain of length $k$, then it is easily seen that $P(a_1,\ldots,a_n) \cong C_{a_1+1} \times_p \cdots \times_p C_{a_n+1}$. For this reason we call $\times_p$ the \textit{proper division product}. Note that $(P_1 \times_p P_2) \times_p P_3 = P_1 \times_p P_2 \times_p P_3$. A natural question is if Theorem \ref{T.shellable} and Proposition \ref{P.nojumps} can be extended to this setting.

Let $P$ and $Q$ be two (pure) shellable posets with unique maximal and unique minimal element.

\begin{question}
Is $\Delta(P \times_p Q)$ shellable?
\end{question}

\begin{question}
Assume $\Delta(P \times_p Q)$ is nonempty.
Is there an integer $t_{P;Q} \geq 0$ such that $H_i(\Delta(P \times_p Q);\ZZ) \neq 0$ if and only if $0 \leq i \leq t_{P;Q}$?
\end{question}

We have tested both questions when $P$ and $Q$ are Boolean lattices on a reasonable sized set of examples. For all those examples the answer to both questions is affirmative. If we denote by $B_i$ the Boolean lattice on $i$ elements, then we have:

\bigskip

\begin{center}
\begin{tabular}{ c | l }
$P \times_p Q$ & $(\rank H_i ( \Delta(P \times_p Q) ) : i\geq0)$ \\ \hline
$B_2 \times_p B_6$ & (15,30,40,30,13,0,\dots) \\ \hline
$B_2 \times_p B_7$ & (17,42,70,70,42,15,0,\dots) \\ \hline
$B_3 \times_p B_6$ & (1,1461,1275,705,172,0,\dots) \\ \hline
$B_3 \times_p B_7$ & (1,3381,3822,2940,1218,232,0,\dots) \\ \hline
\end{tabular}
\end{center}

\bigskip

The structure of the paper is as follows. In Section \ref{S.CL-Shellability} we recall the relation between recursive atom ordering and
CL-shellability and prove Theorem \ref{T.shellable} and Proposition \ref{P.notshellable}.
The rest of the paper deals with the case $n=2$, i.e. the posets $P(a,b)$. In Section \ref{S.Homology} we study the homology of $\Delta ( P(a,b) )$, by characterizing and counting the falling chains of $P(a,b)^\ast$. A first qualitative result concerns the vanishing of the homology, see Propositions \ref{P.vanishing} and \ref{Bounds on homology}. Our proofs of Theorem \ref{T.homology} and Proposition \ref{P.nojumps} depend heavily on the labeling induced by the recursive atom ordering from the proof of Theorem \ref{T.dualCLshellable}.
Moreover, in Corollary \ref{C.contractible} we show that the only poset $P(a,b)$, with $2 \leq a \leq b$, whose order complex is contractible, and indeed collapsible, is $P(3,3)$.
In Section \ref{S.EulerChar} we prove Theorem \ref{T.EulerChar} using generating function techniques. Note that the formula from
Theorem \ref{T.EulerChar} is much simpler than the alternating sum of the rank of the homologies given in Theorem \ref{T.homology}.
From this, in Corollary \ref{C.ZeroEulerChar}, we deduce that $\widetilde \chi ( \Delta ( P(a,b) ) ) = 0$ if $a=b$ and $a$ is odd.

\section{CL-Shellability} \label{S.CL-Shellability}

In this section we prove that the order complex of the poset $ P(a_1,\ldots,a_n)$ is vertex decomposable, hence shellable,
by showing that the dual poset $P(a_1,\ldots,a_n)^\ast$ is $\mathrm{CL}$-shellable.
Indeed, we will show that $P(a_1,\ldots,a_n)^\ast$ admits a recursive atom ordering which, by \cite[Thm. 5.11]{BW96}, is equivalent to
show that the poset is $\mathrm{CL}$-shellable.

Before defining the recursive atom ordering, we need to introduce some more poset terminology. Let $P$ be a poset with order relation $\leq$.  We say that $p \in P$ \textit{covers} $q \in P$, and use the notation $q \rightarrow p$, if $q < p$ and there is no $q' \in P$ with $q < q' < p$. The \textit{atoms} of  a poset $P$ with unique minimal element $\widehat{0}$ are the elements of $P$ that cover $\widehat{0}$.
For $q \leq p$ in $P$, we define the \textit{interval} $[q,p] := \{ q' \in P~:~q \leq q' \leq p\}$, which is a poset with the induced order, unique minimal element $q$ and unique maximal element $p$.
Finally, the \textit{length} of a chain in $P$ is the number of its elements minus one and the \textit{length} of $P$, denoted $\ell(P)$, is the maximal length of its chains.

\smallskip

Some immediate properties of $P(a_1,\dots,a_n)$ are the following:
\begin{itemize}
\item If $a_1,\ldots, a_n \geq 1$ then $P(a_1,\dots,a_n)$ has $a_1 \cdots a_n + 1$ elements, since the elements of $P(a_1,\dots,a_n)$, except the top element, are exactly the elements of the classical divisibility poset with top element $(a_1-1,\dots,a_n-1)$ but with a different partial order;
\item $\ell(P(a_1,\dots,a_n))=\max_{1 \leq i \leq n} \{a_i\}$. In fact, assume that $\max_{1 \leq i \leq n} \{ a_i \} = a_n$. Then the chain
    \[
    (0,\dots,0,0) \rightarrow (0,\dots,0,1) \rightarrow (0,\dots,0,2) \rightarrow \cdots \rightarrow (0,\dots,0,a_n-1) \rightarrow (a_1,\dots,a_{n-1},a_n)
    \]
    has length $a_n$. There are no longer chains since every covering relation is of the form $(c_1,\dots,c_n) \rightarrow (d_1,\dots,d_n)$, where for every $1 \leq i \leq n$, either $c_i = d_i = 0$ or $c_i < d_i$ and $c_j=d_j-1$ for some $1 \leq j \leq n$.
\end{itemize}

\begin{definition}\label{def recursive atom ordering}
  Let $P$ be a poset with unique minimal element $\widehat{0}$ and unique maximal element $\widehat{1}$.
  The poset $P$ admits a \textit{recursive atom ordering} if $\ell(P) \leq 1$ or $\ell(P)>1$ and there is a linear ordering $\preceq$ of the atoms of $P$ satisfying the following two conditions:
\begin{enumerate}
\item[(i)] for all atoms $p$ of $P$, the interval $[p,\widehat{1}]$ admits a recursive atom ordering in which the atoms of $[p,\widehat{1}]$ that belong to $[q,\widehat{1}]$, for some $q \prec p$, come first;
\item[(ii)] for all atoms $p \prec p'$ and elements $p,p' < q$ of $P$, there exist an atom $p'' \prec p'$ of $P$ and an atom $q'$ of $[p',\widehat{1}]$ such that $p''<q'\leq q$.
\end{enumerate}
\end{definition}

Theorem \ref{T.shellable} is a consequence of the following key result.

\begin{theorem} \label{T.dualCLshellable}
  For every $(a_1,\ldots,a_n) \in \mathbb N^n$, the poset $P(a_1,\ldots,a_n)^\ast$ admits a recursive atom ordering and hence is $\mathrm{CL}$-shellable.
\end{theorem}

\begin{proof}
  We denote by $\leq_*$ the order on $P(a_1,\ldots,a_n)^\ast$, that is the dual of the order on $P(a_1,\ldots,a_n)$.

  We proceed by induction on $n$ and $\ell(P(a_1,\ldots, a_n))$.
  If $n=1$, the poset is a single chain, which is easily checked to admit a recursive atom ordering.
  Assume $n>1$. Suppose that $a_i \leq 1$ for some $i$.
  Then $P(a_1,\ldots,a_n)^\ast$ can be identified with $P(a_1,\dots,\widehat{a_i},\dots,a_n)^\ast$, in which we remove the $i$\textsuperscript{th} component from all elements.
  In fact, all the elements of $P(a_1,\dots,a_n)^\ast$, except possibly the bottom element, have the $i$\textsuperscript{th} coordinate equal to zero.
  Thus, by induction on $n$, we know that $P(a_1,\ldots,a_n)^\ast$ admits a recursive atom ordering.

  Hence we may assume that $a_h \geq 2$ for all $h$.

  We will frequently use the following fact.

\begin{itemize}
\item[($\star$)] for every $(b_1,\dots,b_n) \in P(a_1,\dots,a_n)^\ast \setminus \{ (0,\ldots, 0) \}$ and any atom $(c_1,\dots,c_n)$
  of the interval $[(b_1,\dots,b_n),(0,\ldots, 0)]$, there is an index $1 \leq j \leq n$ such that $c_j=b_j-1$.
\end{itemize}

  Now, for every element $(b_1,\dots,b_n)$ of $P(a_1,\dots,a_n)^\ast$, we order the atoms of $[(b_1,\dots,b_n),(0,\ldots, 0)]$ by the dual $\preceq$ of the lexicographic order: for any two atoms $(c_1,\dots,c_n)$ and $(d_1,\dots,d_n)$, we set $(c_1,\dots,c_n) \preceq (d_1,\dots,d_n)$ if and only if either $c_1>d_1$ or there exists $1 \leq i \leq n$ such that $c_h=d_h$ for every $1 \leq h \leq i$ and $c_{i+1}>d_{i+1}$. In particular, the least atom of the interval $[(b_1,\dots,b_n),(0,\ldots,0)]$ is $(\bar b_1,\dots, \bar b_n)$, where $ \bar b_h = b_h-1$ if $b_h \neq 0$ and $\bar b_h=0$ if $b_h=0$.

  For every $(b_1,\ldots, b_n) \in P(a_1,\ldots,a_n)^\ast \setminus \{(a_1,\ldots,a_n)\}$, the interval $[(b_1,\ldots,b_n), (0,\ldots,0)]$ is easily identified with the poset $P(b_1,\ldots, b_n)^\ast$ and hence, by induction on the length, we may assume that the dual lexicographic order is a recursive atom ordering
  for all intervals $[(b_1,\ldots,b_n),(0,\ldots,0)]$ in $P(a_1,\ldots, a_n)^\ast$, with $(b_1,\ldots,b_n) \neq (a_1,\ldots, a_n)$.
  Thus it suffices to verify conditions (i) and (ii) from Definition \ref{def recursive atom ordering} for the ordering of the
  atoms of $P(a_1,\ldots,a_n)^\ast$ only.

  \begin{itemize}
    \item[(i)]
      Let $p_1$ be the least atom of $P(a_1,\dots,a_n)^\ast$. Then $p_1=(a_1-1,\dots,a_n-1)$. Let $(b_1,\dots,b_n) \neq p_1$ be another atom of $P(a_1,\dots,a_n)^\ast$. Notice that, every atom $(c_1,\dots,c_n)$ of the interval\break $[(b_1,\dots,b_n),(0,\ldots,0)]$ satisfies $c_h < b_h \leq a_h-1$ if $b_h \neq 0$, and $c_h=b_h=0 \leq a_h-1$ if $b_h=0$.
      Hence $(c_1,\dots,c_n)$ also belongs to the interval $[p_1, (0,\ldots,0)]$.
      Thus $p_1<(c_1,\dots,c_n)$ and hence condition (i) is fulfilled.
    \item[(ii)]
      Let $p \prec p'$ be atoms of $P(a_1,\dots,a_n)^\ast$. Then
      $p' = (a_1-k_1,\dots,a_n-k_n)$, with $k_h \geq 1$ for every $h=1,\dots,n$. Let $q$ be another element of $P(a_1,\ldots, a_n)^\ast$ such that $p,p' <_\ast q$. Then
      \[
      q=(b_1,\dots,b_n)=(a_1-k_1-k_1',\dots,a_n-k_n-k_n'),
      \]
      where $k_h'=0$ if and only if $a_h-k_h=0$ and $k'_h \geq 1$ otherwise.
      Since $a_h \geq 2$ for every $h$, there exists $s$ such that $a_s-k_s>0$ and hence $k'_s \geq 1$.
      We distinguish between two cases.

      If $k'_t=1$ for some $t$, then we set $q'=q$. Clearly, by ($\star$), $q$ is an atom of $[p',(0,\ldots,0)]$ and $p_1=(a_1-1,\dots,a_n-1) <_* q' \leq_* q$.

      Otherwise, if for every $h$, either $k'_h=0$ or $k'_h>1$, then we set $q'=(c_1,\dots,c_n)$, where for every $h=1,\dots,n$,
      \[
        c_h =
          \begin{cases}
            0 & \text{if } k'_h=0 \\
            a_h-k_h-k_h'+\min_{1 \leq r \leq n} \{ k_r' : k'_r \neq 0 \}-1 & \text{if } k'_h>1
          \end{cases}.
      \]

     Notice that $k'_s \neq 0$. Again $q'$ is an atom of $[p',(0,\dots,0)]$ by ($\star$), since $c_h=0$ when $a_h-k_h=0$, $c_h \leq a_h-k_h-1$ when $k'_h>1$, and $c_t=a_t-k_t-1$, where $t$ is such that $k'_t=\min_{1 \leq r \leq n} \{ k_r' : k'_r \neq 0 \}$.

     On the other hand, $q' \leq_* q$, since $c_h=b_h=0$ if $b_h=0$, and $\min_{1 \leq r \leq n} \{ k_r' : k'_r \neq 0 \}-1 > 0$, hence $b_h<c_h$, if $b_h>0$.
     Furthermore, $p'':=p_1 \prec p'$ because $c_h < a_h-1$ for every $h$ (since either $c_h=0$, if
     $k'_h=0$, or $-k_h-k_h'+\min_{1 \leq r \leq n} \{ k_r' : k'_r \neq 0 \} \leq -1$, if $k'_h>1$) and $p'' <_\ast q'$. This concludes the proof. \qedhere
    \end{itemize}
\end{proof}

Now we can prove Proposition \ref{P.notshellable} showing that, indeed, Theorem \ref{T.dualCLshellable} does not hold for $P(a_1,\dots,a_n)$.

\begin{proof}[Proof of Proposition \ref{P.notshellable}]
  We prove that the poset $P(4,4)$ in Figure \ref{F.nonCL} does not admit any recursive atom ordering, by showing that no ordering on the atoms of $P(4,4)$ does fulfill condition (ii) of Definition \ref{def recursive atom ordering}.
  Let $\preceq$ be a linear order on the atoms $(1,0), (1,1), (0,1)$ of $P(4,4)$. Since $P(4,4)$ is
  invariant under switching coordinates, we may assume that $p=(1,0) \preceq p'=(0,1)$. We consider $q=(2,3)$. Clearly $p,p'<q$.
  For every $p'' \prec p'$ and for every atom $q'$ of $[p',\widehat 1]$, either $p'' \nless q'$ or $q' \nleqslant q$.

  Thus $P(4,4)$ does not admit any recursive atom ordering and hence is not $\mathrm{CL}$-shellable.
  Nevertheless, by Theorem \ref{T.dualCLshellable}, we know that $P(4,4)^*$ is $\mathrm{CL}$-shellable and hence its order complex $\Delta(P(4,4)^\ast) \cong \Delta(P(4,4))$ is shellable.
\end{proof}

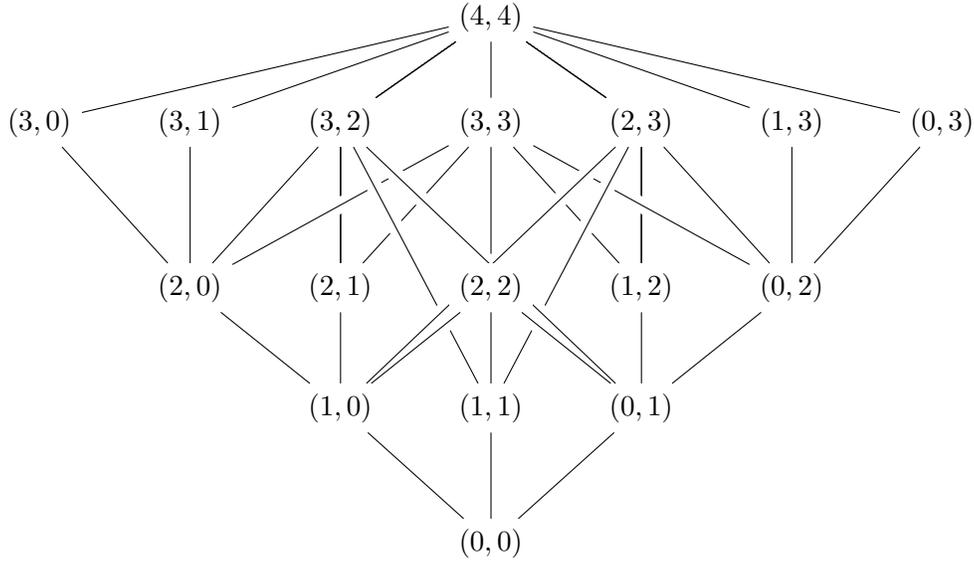
\begin{figure}[ht!]
\begin{tikzpicture}
  \node (max) at (0,5.6) {$(4,4)$};
  \node (o) at (6,4.2) {$(0,3)$};
  \node (n) at (4,4.2) {$(1,3)$};
  \node (m) at (2,4.2) {$(2,3)$};
  \node (l) at (0,4.2) {$(3,3)$};
  \node (k) at (-2,4.2) {$(3,2)$};
  \node (j) at (-4,4.2) {$(3,1)$};
  \node (i) at (-6,4.2) {$(3,0)$};
  \node (h) at (4,2) {$(0,2)$};
  \node (g) at (2,2) {$(1,2)$};
  \node (f) at (0,2) {$(2,2)$};
  \node (e) at (-2,2) {$(2,1)$};
  \node (d) at (-4,2) {$(2,0)$};
  \node (c) at (2,0.4) {$(0,1)$};
  \node (b) at (0,0.4) {$(1,1)$};
  \node (a) at (-2,0.4) {$(1,0)$};
  \node (min) at (0,-1.4) {$(0,0)$};
  \draw (min) -- (a) -- (d) -- (i) -- (max) -- (l) -- (f) -- (b) -- (min) -- (c) -- (h) -- (o) -- (max)
  (a) -- (e) -- (k) -- (max) -- (m) -- (g) -- (c)
  (d) -- (j) -- (max) -- (n) -- (h) -- (m) -- (max) -- (k) -- (d)
  (e) -- (l) -- (g)
  (e) -- (k)
  (g) -- (m);
  \draw[preaction={draw=white, -,line width=6pt}] (d) -- (l) -- (h);
  \draw[preaction={draw=white, -,line width=6pt}] (m) -- (b) -- (k);
  \draw[preaction={draw=white, -,line width=6pt}] (c) -- (f) -- (a)
  (a) -- (m)
  (k) -- (c);
  \node[rectangle,fill=white] at (0,2) {$(2,2)$};
\end{tikzpicture}
\caption{The poset $P(4,4)$} \label{F.nonCL}
\end{figure}

\section{Homology of the order complex of $P(a,b)$} \label{S.Homology}

In this section we study the simplicial homology groups of $\Delta (  P(a,b) ) \cong \Delta (  P(a,b)^* )$ with coefficients in $\mathbb Z$. Without loss of generality, we may assume $a \leq b$.

For this we construct, from a recursive atom ordering of a poset $P$ with unique minimal element $\widehat{0}$ and unique maximal element $\widehat{1}$, a
labeling $\lambda(q \rightarrow p)$ of the cover relations of $P$ by integers. This construction is contained in the proof of \cite[Thm. 3.2]{BW83}, showing that a poset with recursive atom ordering is $\mathrm{CL}$-shellable. Note that \cite{BW83} deals with pure posets only but, as noted in \cite{BW96}, the same construction goes through in the non-pure case. For our purposes we only need a part of this construction.

\begin{construction}
  \label{recursiveCL}
  Let $P$ be a poset with unique minimal element $\widehat{0}$ and unique maximal element $\widehat{1}$. Assume that $P$ admits a recursive atom ordering and $\preceq$ is the linear ordering of the atoms of $P$. First we choose a labeling $\lambda$ of the cover relations $\widehat{0} \rightarrow p$
  by integers such that $\lambda(\widehat{0} \rightarrow p) < \lambda(\widehat{0} \rightarrow p')$ if $p \prec p'$.
  If $\ell(P)>1$, let $F(p')$ be the set of all atoms of $[p',\widehat{1}]$ that cover some atom $p \prec p'$ of $P$.
  Then we choose the labeling as follows:
  \begin{gather*}
    \mbox{ if } q \in F(p'), \mbox{ then } \lambda(p' \rightarrow q) < \lambda(\widehat 0 \rightarrow p'),\\
    \mbox{ if } q\notin F(p'), \mbox{ then } \lambda(p' \rightarrow q) > \lambda(\widehat 0 \rightarrow p').
  \end{gather*}
\end{construction}

By \cite[Thm. 5.9]{BW96}, given a labeling $\lambda$ on $P$, from Construction \ref{recursiveCL}:

\begin{itemize}
\item[(FCH)] The rank of the $i$\textsuperscript{th} reduced homology group
  of $\Delta(P)$ with integer coefficients equals the number of
  chains $\widehat{0} = p_0 \rightarrow p_1 \rightarrow \cdots
  \rightarrow p_{i+1} \rightarrow p_{i+2} = \widehat{1}$ of length $i+2$ in $P$ for which
  $\lambda(p_0 \rightarrow p_1) > \cdots > \lambda(p_{i+1} \rightarrow p_{i+2})$.
\end{itemize}

The latter chains are called \textit{falling}. We use this principle for
determining the reduced homology groups of a poset with recursive atom ordering and we refer to it as
the \textit{Falling-Chain-Homology principle} or (FCH) for short.

In particular, $\rank H_0 ( \Delta(P); \ZZ )$ is one more than the number of falling chains of length $2$.

\smallskip
From now on, we consider $P(a,b)^*$ equipped with the recursive atom ordering from
Theorem \ref{T.dualCLshellable}. Let $\lambda$ be the labeling of the edges
of $P(a,b)^*$ induced through Construction \ref{recursiveCL} by this ordering.

We call $(c,d) \in \mathbb N^2$ a \textit{border elements} of $P(a,b)^\ast$ if
it has one of the forms $(1,k)$, $(k,1),(0,k)$ or $(k,0)$ for some $k \geq 2$.

\begin{lemma}\label{Falling chains}
 Let $m:\widehat{0} = (a,b) = p_0 \rightarrow p_1 \rightarrow \cdots \rightarrow p_k = (0,0) = \widehat{1}$ be
 a chain of length $k \geq 2$ in $P(a,b)^\ast$.  Then $m$ is falling with respect to $\lambda$ if and only if it satisfies the
 following conditions:
 \begin{enumerate}
   \item[(i)] $p_{i+1}$ is not the lexicographically least atom of $[p_i,(0,0)]$ for $0 \leq i \leq k-2$,
   \item[(ii)] $p_i$ is not a border element for $1 \leq i \leq k-2$.
 \end{enumerate}

 In particular, by condition {\rm (i)}, if $m$ is falling, $p_i=(c_i,d_i)$ is not a border element and $i \leq k-2$, then $p_{i+1}$ cannot be the element $(c_i-1,d_i-1)$.
\end{lemma}

\begin{proof}
  Let $m$ be a falling chain of length $k$ in $P(a,b)^\ast$. We have to verify (i) and (ii).
  \begin{itemize}
    \item[(i)]  Suppose that, for some $0 \leq i \leq k-2$, $p_{i+1}$ is the least atom of $[p_i,(0,0)]$.
      Then $\lambda(p_{i+1} \rightarrow p_{i+2})>\lambda(p_i \rightarrow p_{i+1})$ by Construction \ref{recursiveCL}.
      Then $m$ is not a falling chain.
    \item[(ii)]
      Suppose that $p_i$ is a border element for some $1 \leq i \leq k-2$.
      Then $p_{i+1}$ is the only atom of $[p_i,(0,0)]$. In particular, it is the least atom, which contradicts (i).
  \end{itemize}

  Conversely, suppose that $m$ satisfies conditions (i) and (ii). Then it is immediate from Construction \ref{recursiveCL} that
  $\lambda(p_{i+1} \rightarrow p_{i+2})<\lambda(p_i \rightarrow p_{i+1})$ for every $0 \leq i \leq k-2$. Hence $m$ is a falling chain.
\end{proof}

We now describe the homology of $\Delta (  P(a,b) )$. Let
\[
t_{(a,b)}= \max \big\{ i : H_i ( \Delta (  P(a,b) ); \mathbb Z ) \neq 0 \big\}.
\]
By convention we set $t_{(a,b)}=-1$ if $H_i ( \Delta (  P(a,b) ); \mathbb Z ) = 0$
for every $i \geq 0$, i.e. $\Delta(P(a,b))$ is empty.

\begin{example} \label{E.a=2,3}
  First we describe the order complex $\Delta(P(a,b))$ in some simple cases, when $0 \leq a \leq 3$.

  \begin{enumerate}
    \item Let $a=0$ or $a=1$.

       If $b \leq 1$ then $ P(a,b)=\emptyset$ and $t_{(a,b)}=-1$. Hence assume $b \geq 2$. Then $P(a,b)$ is a single chain of length $b$
       \[
          (0,0) \rightarrow (0,1) \rightarrow (0,2) \rightarrow \cdots \rightarrow (0,b-1) \rightarrow (a,b)
       \]
       and its order complex is a simplex of dimension $b-2$. Thus $t_{(a,b)}=0$ if $b \geq 2$.

   \item Let $a=2$.

       In Figure \ref{F.a=2} we draw the poset $P(2,b)^\ast$. For $b=2$, the poset $P(2,2)^\ast$ has only three maximal chains,
        $(2,2) \rightarrow (1,1) \rightarrow (0,0)$, $(2,2) \rightarrow (1,0) \rightarrow (0,0)$, $(2,2) \rightarrow (0,1) \rightarrow (0,0)$. By Lemma \ref{Falling chains}, the first chain is not falling and the other two are falling. Moreover, the order complex $\Delta (  P(2,2)^\ast )$ consists of three isolated points.

       Let $b>2$.
       There are exactly two maximal chains of length $2$ in $P(2,b)^\ast$, $m: (2,b) \rightarrow (1,0) \rightarrow (0,0)$ and
       $m':(2,b) \rightarrow (1,1) \rightarrow (0,0)$.
       Note, that the least atom of $[(2,b),(0,0)]$ is $(1,b-1)$. Since $b-1>1$, neither of $(1,0)$ and $(1,1)$ is the least atom of $[(2,b),(0,0)]$. Hence both
       $m$ and $m'$ are falling chains.

       Now we show that no other maximal chain in $ P(2,b)^\ast$ is falling. Consider the maximal chain
       $m_t: (2,b) \rightarrow (1,t) \rightarrow (0,t-1) \rightarrow \cdots \rightarrow (0,1) \rightarrow (0,0)$, with $2 \leq t \leq b-1$.
       Since $(1,t)$ is a border element, the chain $m_t$ is not falling by Lemma \ref{Falling chains}.
       Again, by Lemma \ref{Falling chains}, the remaining chain $(2,b) \rightarrow (0,b-1) \rightarrow (0,b-2) \rightarrow \cdots \rightarrow (0,1) \rightarrow (0,0)$ is not falling, since $(0,b-1)$ is a border element. Thus, by (FCH), for every $b \geq 2$,
       \[
         \rank H_0 ( \Delta (  P(2,b)^\ast ); \mathbb Z ) = 3 \text{ and }
         \rank H_i ( \Delta (  P(2,b)^\ast ); \mathbb Z ) = 0, \text{ for every } i \geq 1.
       \]
       Hence $t_{(2,b)} = 0$, for every $b \geq 2$.

       \begin{figure}[ht!]
         \begin{tikzpicture}
           \node (max) at (0,6) {$(0,0)$};
           \node (e) at (0,5) {$(0,1)$};
           \node (d) at (0,4) {$\vdots$};
           \node (c) at (0,3) {$(0,b-3)$};
           \node (b) at (0,2) {$(0,b-2)$};
           \node (a) at (0,1) {$(1,b-1)$};
           \node (min) at (0,-0.7) {$(2,b)$};
           \node (f) at (2,1) {$(0,b-1)$};
           \node (g) at (-2,1) {$(1,b-2)$};
           \node (h) at (-3.3,1) {$\cdots$};
           \node (i) at (-4.6,1) {$(1,2)$};
           \node (j) at (-5.9,1) {$(1,1)$};
           \node (k) at (-7.2,1) {$(1,0)$};
           \draw (min) -- (a) -- (b) -- (c) -- (d) -- (e) -- (max)
           (min) -- (f) -- (b)
           (min) -- (g) -- (c)
           (min) -- (i) -- (e)
           (min) -- (j) -- (max)
           (min) -- (k) -- (max);
         \end{tikzpicture}
         \caption{The poset $P(2,b)^\ast$} \label{F.a=2}
       \end{figure}
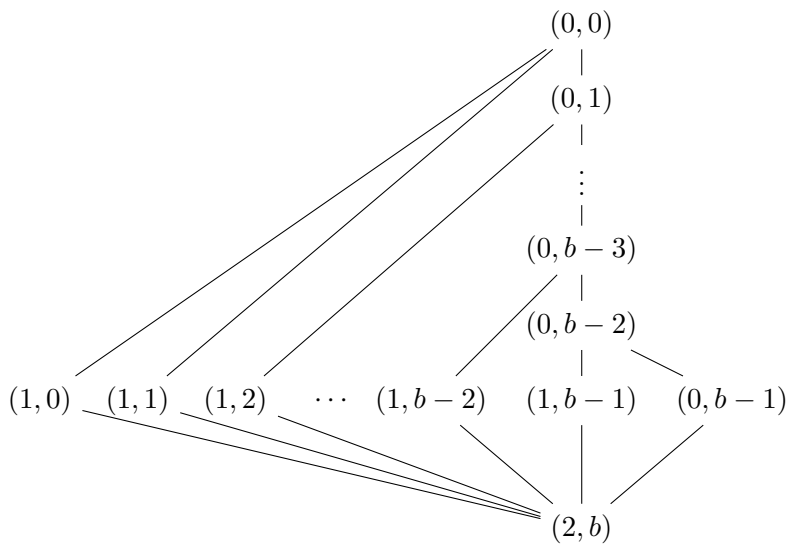

    \item Let $a=3$.

      If $b=3$, there are no falling chains by Lemma \ref{Falling chains}. Hence, by (FCH),
      \[
        \rank H_i ( \Delta (  P(3,3)^\ast ); \mathbb Z ) =0 \text{ if and only if } i \neq 0.
      \]
      Thus $t_{(3,3)} = 0$.

      We may assume $b>3$. Notice that, there are no maximal chains of length $2$.
      Therefore (FCH) implies $\rank H_0 ( \Delta (  P(3,b)^\ast ); \mathbb Z ) =1$.
      In fact, a chain of the form $(3,b) \rightarrow (c,d) \rightarrow (0,0)$ is not maximal for every
      $0 \leq c \leq 2$, $0 \leq d \leq b-1$ and with $c=2$ or $d=b-1$.

      We consider the maximal chains of length $\geq 3$. By Lemma \ref{Falling chains}, if a chain
      $m: p_0=(3,b) \rightarrow p_1 \rightarrow \cdots \rightarrow p_k=(0,0)$ is falling, then
      $p_1$ is not one of the elements $(2,b-1)$ (least atom of $[(3,b),(0,0)]$) and $(1,b-1), (0,b-1), (2,1), (2,0)$ (border elements).
      In other words, if $m$ is falling, then it is of the
      form $m: p_0=(3,b) \rightarrow p_1=(2,t) \rightarrow \cdots \rightarrow p_k=(0,0)$, with $2 \leq t \leq b-2$.

      If $t=2$, there are three maximal chains containing $(2,2)$: $m_1: (3,b) \rightarrow (2,2) \rightarrow (1,1) \rightarrow (0,0)$,
      $m_2: (3,b) \rightarrow (2,2) \rightarrow (1,0) \rightarrow (0,0)$ and $m_3: (3,b) \rightarrow (2,2) \rightarrow (0,1) \rightarrow (0,0)$.
      By Lemma \ref{Falling chains}, $m_1$ is not falling, since $(1,1)$ is the least atom of $[(2,2),(0,0)]$, while $m_2$ and $m_3$ are falling.

      Let $t \geq 3$. If the third element $p_2$ in the chain $m$ is of the form $(0,t-1)$, it is a border element and $m$ is not falling.
      By Lemma \ref{Falling chains} we can exclude $p_2 = (1,t-1)$ since it is the least atom of $[(2,t),(0,0)]$.
      Assume $p_2$ is of the form $(1,h)$, with $0 \leq h \leq t-2$.
      Thus $h=0$ or $h=1$, since for $h \geq 2$ the element $(1,h)$ is a border element. The element $p_3$ of $m$ is now forced to be $(0,0)$.
      By Lemma \ref{Falling chains}, all maximal chains of the form $(3,b) \rightarrow (2,t) \rightarrow (1,h) \rightarrow (0,0)$,
      with $3 \leq t \leq b-2$ and $h=0,1$ are falling.
      Then we have $2+2(b-2-3+1)=2(b-3)$ falling chains of length $3$. Moreover, there are no falling chains of length $\geq 4$.
      Thus (FCH) implies $\rank H_1 ( \Delta (  P(3,b)^\ast ); \mathbb Z ) =2(b-3)$ and
      $\rank H_i ( \Delta (  P(3,b)^\ast ); \mathbb Z ) =0$, for $i \geq 2$. Hence $t_{(3,b)} = 1$ for every $b \geq 4$.
  \end{enumerate}
\end{example}

\begin{proposition} \label{P.vanishing}
  For every $2 \leq a \leq b$, the reduced homology
   $\widetilde H_i(\Delta (  P(a,b) ); \mathbb Z )$ is zero whenever $i> a-2$.
\end{proposition}

\begin{proof}
  If $a=2,3$, the claim follows from Example \ref{E.a=2,3}. Hence we may assume $a \geq 4$.
  By (FCH), it suffices to show that there are no falling chains of length $\ell \geq a+1$.
  If $b=a$, then there are no chains of length $\ell \geq a+1$, since $\ell(P(a,b)^\ast)=\ell(P(a,b))=a$ (see discussion after Definition \ref{def recursive atom ordering}).

  Let $b \geq a+1$ and $m: p_0=(a,b) \rightarrow p_1 \rightarrow \cdots \rightarrow p_{\ell-1} \rightarrow p_\ell = (0,0)$
  be a maximal chain of length $\ell \geq a+1$.
  We denote the $i$\textsuperscript{th} element of $m$ by $p_i=(c_i,d_i)$.
  Notice that, for every $i$, if $c_i \neq 0$, then $c_{i+1} < c_i$, otherwise if $c_i=0$, then $c_{i+1}=0$.
  Moreover, $0 \leq c_{i+1} \leq c_i \leq \max\{ 0,a-i \}$, for every $i$. In particular, $c_{\ell -1} \leq 0$, hence $c_{\ell -1}=0$.
  Thus $d_{\ell-1}=1$, otherwise $m$ is not maximal. On the other hand, $0 \leq c_{\ell-2} \leq 1$, hence $d_{\ell-2} \geq 2$,
  otherwise $c_{\ell-2} \nleq c_{\ell-1}$. Therefore, $p_{\ell-2}$ is a border element and $m$ is not a falling chain.
\end{proof}

%
%

The following result, together with Corollary \ref{C.LastNonzeroHomology}, shows that the
highest degree in which the homology of $\Delta (  P(a,b) )$ is non-zero depends on
the value of $b$ relative to $a$.

\begin{proposition}\label{Bounds on homology}
  Let $4 \leq a \leq b$. If $2a-3k-2 \leq b \leq 2a-3k$ for some $k \geq 1$, then
  \[
     t_{(a,b)}=a-2-k.
  \]
\end{proposition}

\begin{proof}
  Let $2a-3k-2 \leq b \leq 2a-3k$ for some $k \geq 1$. Then the chain
  \begin{gather*}
    m: \widehat{0}\!=\!(a,b) \rightarrow (a\!-\!1,b\!-\!2) \rightarrow (a\!-\!2,b\!-\!4) \rightarrow\!
       \cdots\! \rightarrow (2k\!+\!2,b\!-\!2a\!+\!4k\!+\!4) \rightarrow\\ \rightarrow (2k,b\!-\!2a\!+\!4k\!+\!3)
       \rightarrow (2k\!-\!2,b\!-\!2a\!+\!4k\!+\!2) \rightarrow\! \cdots\! \rightarrow (2,b\!-\!2a\!+\!3k\!+\!4)
       \rightarrow (1,0) \rightarrow (0,0)\!=\!\widehat{1}
  \end{gather*}
  is maximal since $2k \geq 2$ and $2k+2 \leq a$. To show the second inequality, it is enough to notice that $k \leq \lfloor \frac{a}{3} \rfloor$. Moreover, $m$ is falling by Lemma \ref{Falling chains} and has length $a-k-2$. Note that $b-2a+3k+4 \geq 2$, since $b \geq 2a-3k-2$ by assumption.
  Thus, by (FCH), $t_{(a,b)} \geq a-2-k$.

  Conversely, we show that $t_{(a,b)}\leq a-2-k$. Again by (FCH), it is enough
  to show that no maximal chain of length $\ell \geq a-k+1$ in $P(a,b)^\ast$ is falling.

  Let $m : (a,b) = p_0 \rightarrow \cdots \rightarrow p_\ell = (0,0)$ be a maximal chain
  in $P(a,b)^\ast$ of length $\ell \geq a-k+1$. We denote the $i$\textsuperscript{th} element of $m$ by $p_i=(c_i,d_i)$. Furthermore,
  for $i=1,\ldots,\ell$, let us denote by $u_i=c_{i-1}-c_i$ and $v_i=d_{i-1}-d_i$ the increment on the first and on the
  second component respectively. Since $m$ is a maximal chain,
  $p_{\ell-1}$ is one of the elements $(1,1), (1,0)$ and $(0,1)$. Hence, we have one of the following three cases:
  \begin{enumerate}
    \item[(i)] $u_\ell=1$ and $v_\ell=1$,
    \item[(ii)] $u_\ell=0$ and $v_\ell=1$,
    \item[(iii)] $u_\ell=1$ and $v_\ell=0$.
  \end{enumerate}

  Note that the cases (ii) and (iii) are symmetric, hence it is enough to discuss (i) and (ii).

\smallskip
  \noindent {\sf Claim:}
    For every $i=1,\dots,\ell-1$, $u_i>0$ and $v_i>0$.

  \noindent $\triangleleft$ In fact, if $u_i=0$ for some
  $1 \leq i \leq \ell-1$, then $v_i \geq 1$ and $c_{i-1}=c_i$.
  This only happens if $c_i=0$, hence $p_{i-1}=(0,d_{i-1})$ is a border element
  and $m$ is not a falling chain by Lemma \ref{Falling chains}. By symmetry, also $v_i>0$
  for $i=1,\dots,\ell-1$. $\triangleright$
\smallskip

  First assume that (ii) holds, thus $u_\ell=0$ and $v_\ell=1$. If $u_{\ell-1}=0$ or $u_{\ell-1}=1$, then $c_{\ell-2}=c_{\ell-1}=0$ or $c_{\ell-2}=c_{\ell-1}+1=1$.
  Hence $m$  is not a falling chain by Lemma \ref{Falling chains}, since it contains a border element. Thus $u_{\ell-1}\geq 2$ and this implies $c_{\ell-2}\geq 2$.
  By the maximality of $m$, it follows that $v_{\ell-1}=1$. Hence
  \[
    \sum_{i=1}^{\ell-2} u_i \leq a-2 \mbox{ and } \sum_{i=1}^{\ell-2} v_i = b-2.
  \]

  Summing the two expressions, we have
  \[
    \sum_{i=1}^{\ell-2} (u_i+v_i) \leq a+b-4 \leq 3a-3k-4 < 3(a-k-1) \leq 3(\ell-2).
  \]

  If $u_i+v_i \geq 3$ for every $i=1,\dots,\ell-2$, then we have
  \[
    3(\ell-2) \leq \sum_{i=1}^{\ell-2} (u_i+v_i) < 3(\ell-2),
  \]
  which is a contradiction. Thus, since both $u_i$ and $v_i$ are positive,
  there exists $1 \leq j \leq \ell-2$ such that $u_j+v_j=2$ and then $u_j=v_j=1$.
  This implies that the element $p_j$ in $m$ is the least
  atom of $[p_{j-1},(0,0)]$, and thus $m$ is not a falling chain by Lemma \ref{Falling chains}.

  For the case (i), using a similar argument, one shows that there exists $1 \leq j \leq \ell-2$
  such that $u_j+v_j=2$ and then $u_j=v_j=1$. This concludes the proof.
\end{proof}

\begin{corollary} \label{C.contractible}
  Let $2 \leq a \leq b$. Then $\Delta( P(a,b))$ is contractible if and only if $a=b=3$.
  Moreover, $\Delta( P(3,3))$ is collapsible.
\end{corollary}

\begin{proof}
  The assertion follows from Example \ref{E.a=2,3}, Proposition \ref{P.vanishing} and Proposition \ref{Bounds on homology}.
  Indeed, the order complex $\Delta (  P(3,3) )$ is a connected acyclic graph, as shown in Figure \ref{F.ord(3,3)}.
  Hence it is collapsible.
  \begin{figure}[ht!]
  \tikzstyle{every picture}=[line width=.65pt,minimum size=3pt,every label/.append style={font=\small},label distance=-4pt]
    \begin{tikzpicture}[scale=0.8]
      \node[label={above:$(0,2)$}] (h) at (4,2) {};
      \node[label={above:$(1,2)$}] (g) at (2,2) {};
      \node[label={above:$(2,2)$}] (f) at (0,2) {};
      \node[label={above:$(2,1)$}] (e) at (-2,2) {};
      \node[label={above:$(2,0)$}] (d) at (-4,2) {};
      \node[label={below:$(0,1)$}] (c) at (2,0) {};
      \node[label={below:$(1,1)$}] (b) at (0,0) {};
      \node[label={below:$(1,0)$}] (a) at (-2,0) {};
      \draw [fill=black] (4,2) circle (1.8pt);
      \draw [fill=black] (2,2) circle (1.8pt);
      \draw [fill=black] (0,2) circle (1.8pt);
      \draw [fill=black] (-2,2) circle (1.8pt);
      \draw [fill=black] (-4,2) circle (1.8pt);
      \draw [fill=black] (2,0) circle (1.8pt);
      \draw [fill=black] (0,0) circle (1.8pt);
      \draw [fill=black] (-2,0) circle (1.8pt);
      \draw (-4,2) -- (-2,0) -- (0,2) -- (2,0) -- (4,2)
      (-2,0) -- (-2,2)
      (2,2) -- (2,0)
      (0,0) -- (0,2);
    \end{tikzpicture}
    \caption{The order complex of $P(3,3)$} \label{F.ord(3,3)}
  \end{figure}
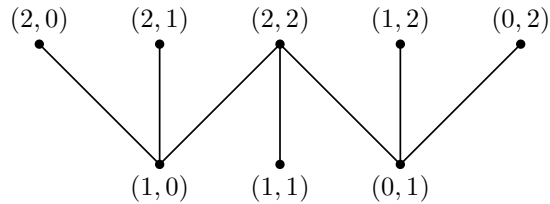
\end{proof}

\begin{lemma}
\label{increments}
Let $m:\widehat{0} = (a,b) = p_0 \rightarrow \cdots \rightarrow p_i=(c_i,d_i) \rightarrow \cdots \rightarrow p_{k+1} = (0,0) = \widehat{1}$ be a falling chain in $P(a,b)^\ast$. Then $p_k$ is one of the elements $(1,0), (1,1)$ and $(0,1)$. Moreover, if, for every $i=1,\ldots,k+1$, we set $u_i=c_{i-1}-c_i$ and $v_i=d_{i-1}-d_i$, then:
\begin{enumerate}
\item[(i)] if $p_k=(1,0)$, then $u_k=1$ and $v_k \geq 2$,
\item[(ii)] if $p_k=(0,1)$, then $u_k \geq 2$ and $v_k=1$.
\end{enumerate}
\end{lemma}

\begin{proof}
Clearly any falling chain contains exactly one of the elements $(1,1), (1,0)$ and $(0,1)$, otherwise it is not maximal in $P(a,b)^\ast$.
\begin{enumerate}
  \item[(i)] Let $p_k=(1,0)$. If $u_k \geq 2$, then $v_k=1$, otherwise the chain is not maximal. This implies that $p_{k-1}=(c_{k-1},1)$, with $c_{k-1} \geq 3$, hence the chain is not maximal because $(c_{k-1},1) < (2,0) < (1,0)$. Thus $u_k=1$ and $v_k \geq 2$.
\end{enumerate}
Similarly one shows (ii).
\end{proof}

In particular, notice that, in both cases of Lemma \ref{increments} there is no further restriction on the other increments $u_i$ and $v_i$, with $i \leq k-1$.

Now we are in position to prove Theorem \ref{T.homology}.

\begin{proof}[Proof of Theorem \ref{T.homology}]
  Let $2\leq a\leq b$ and $0\leq i\leq a-2$.
  By Lemma \ref{increments}, any falling chain
  contains exactly one of the elements $(1,1), (1,0)$ and $(0,1)$.
  Hence, by (FCH), $\rank \widetilde H_i (\Delta (  P(a,b) ); \mathbb Z ) = F_{(1,1)} + F_{(1,0)} + F_{(0,1)}$,
  where $F_{(c,d)}$ denotes the number of falling chains in $P(a,b)$ of length $i+2$ containing the element $(c,d)$.
  Before computing the three contributions we need one more general fact.

  Let  $m: p_0=(a,b) \rightarrow p_1 \rightarrow \cdots \rightarrow p_{i+1} \rightarrow p_{i+2} = (0,0)$
  be a falling chain of length $i+2$. We set $p_\ell=(c_\ell,d_\ell)$ for $0 \leq \ell \leq i+2$
  and $(u_\ell,v_\ell) = p_{\ell-1}-p_\ell$ for $1 \leq \ell \leq i+2$.

  Clearly,
  \begin{equation}\label{Eq.u_h-v_h}
    \sum_{\ell=1}^{i+2} u_\ell=a \text{ and } \sum_{\ell=1}^{i+2} v_\ell=b.
  \end{equation}

  We recall that, for every $1 \leq \ell \leq i+1$, $u_\ell>0$ and $v_\ell>0$ (as in the proof of Proposition \ref{Bounds on homology}).

  We set
  \[
    S=\{\ell : u_\ell=1\},\,\, T=\{\ell : u_\ell \geq 2\} \text{ and } U=\{\ell : v_\ell=1\},\,\, V=\{\ell : v_\ell \geq 2\}
  \]
  and $s=|S|$, $t=|T|$. Notice that $s+2t \leq a$.
  By Lemma \ref{Falling chains}, we have
  \begin{eqnarray}
     \label{e1}
     S \cap U & \subseteq & \{i+2\}.
  \end{eqnarray}
  Since, for $1 \leq \ell \leq i+1$, $p_{\ell-1} < p_{\ell}$
  is a cover relation, it follows that for all $1 \leq \ell \leq i+1$ either
  $u_\ell = 1$ or $v_\ell = 1$ (see also ($\star$)).
  Thus
  \begin{equation}\label{e2}
  \begin{array}{r@{}l}
     S \cap \{1,\ldots,i+1\}\ & = V \\
     U \cap \{1,\ldots,i+1\}\ & = T
  \end{array}
  \end{equation}

  Now we compute the numbers $F_{(1,1)},F_{(1,0)}$ and $F_{(0,1)}$.

  \smallskip

  \noindent $\rightarrow \,\,F_{(1,1)}$

  For contributions to $F_{(1,1)}$, we have $u_{i+2}=1$ and $v_{i+2}=1$. Then
  by the claim above
  \begin{equation*}\label{Eq.(s+t)1}
    s+t=i+2=|U|+|V|.
  \end{equation*}
  By \eqref{e1}, it follows that $S \cap U = \{i+2\}$.
  By \eqref{e2}, we have $S = V \cup \{i+2\}$ and $T \cup \{i+2\} = U$.
  Thus, fixing $S$ fixes the other sets.

  For $S$ we have ${i+1 \choose t}$ choices. Once $S$ is fixed, we have
  ${a-i-3 \choose t-1}$ choices for the $u_\ell$, with $\ell \in T$.
  Now we are left with ${b-i-3 \choose i-t}$ choices for the $v_\ell$, with $\ell \in V$.
  By Lemma \ref{Falling chains}, each choice corresponds to a falling chain.

  This sums up to
  \[
    F_{(1,1)} = \sum_{t=0}^{a-i-2} \binom{b-i-3}{i-t} \binom{a-i-3}{t-1} \binom{i+1}{t}.
  \]

  \smallskip

  \noindent $\rightarrow \,\,F_{(1,0)}$

  For contributions to $F_{(1,0)}$, we have $u_{i+2}=1$ and $v_{i+2}=0$. Thus
  \begin{equation*}\label{Eq.(s+t)2}
    s+t=i+2=|U|+|V|+1.
  \end{equation*}
  By \eqref{e1}, we have $S \cap U = \emptyset$ and, by \eqref{e2},
  $S \setminus \{i+2\} = V$ and $T = U$.
  Again, fixing $S$ fixes the other sets. But there is one additional constraint.
  By Lemma \ref{increments} (i), we have $u_{i+1} = 1$ and $v_{i+1} \geq 2$.
  Thus $i+1 \in S \cap V$.

  For $S$ we have ${i \choose t}$ choices. Once we have $S$ fixed, we have
  ${a-i-3 \choose t-1}$ choices for the $u_\ell$, with $\ell \in T$.
  Now we are left with ${b-i-2 \choose i-t}$ choices for the $v_\ell$, with $\ell \in V$.

  This sums up to
  \[
    F_{(1,0)} = \sum_{t=0}^{a-i-2} \binom{b-i-2}{i-t} \binom{a-i-3}{t-1} \binom{i}{t}.
  \]

  \smallskip

  \noindent $\rightarrow \,\,F_{(0,1)}$

  For contributions to $F_{(0,1)}$, we have $u_{i+2} = 0$ and $v_{i+2}=1$. Thus
  \begin{equation*}
    s+t+1=i+2=|U|+|V|.
  \end{equation*}
  By \eqref{e1}, we have $S \cap U = \emptyset$ and, by \eqref{e2},
  $S = V$ and $T \cup \{i+2\} = U$.
  Again, fixing $S$ fixes the other sets. But also here there is one additional constraint.
  By Lemma \ref{increments} (ii), we have $u_{i+1} \geq 2$ and $v_{i+1} \geq 1$.
  Thus $i+1 \in T \cap U$.

  For $S$ we have ${i \choose t-1}$ choices. Once we have $S$ fixed, we have
  ${a-i-2 \choose t-1}$ choices for the $u_\ell$, with $\ell \in T$.
  Now we are left with ${b-i-3 \choose i-t}$ choices for the $v_\ell$, with $\ell \in V$.

  Since $t \geq 1$, this sums up to
  \begin{eqnarray*}
    F_{(0,1)} & = & \sum_{t=1}^{a-i-1} \binom{b-i-3}{i-t} \binom{a-i-2}{t-1} \binom{i}{t-1}
               =  \sum_{t=0}^{a-i-2} \binom{b-i-3}{i-t-1} \binom{a-i-2}{t} \binom{i}{t}.
  \end{eqnarray*}

  Then, summing these three contributions, we obtain:
  \begin{gather}\label{Eq.homology1}
    \mathrm{rank}\ \widetilde H_i (\Delta (  P(a,b) ); \mathbb Z ) = F_{(1,1)} \!+\! F_{(1,0)} \!+\! F_{(0,1)} \\
    = \sum_{t=0}^{a-i-2} \bigg[ \binom{b\!-\!i\!-\!3}{i\!-\!t} \binom{a\!-\!i\!-\!3}{t\!-\!1} \binom{i\!+\!1}{t} \!+\!
    \binom{b\!-\!i\!-\!2}{i\!-\!t} \binom{a\!-\!i\!-\!3}{t\!-\!1} \binom{i}{t} \!+\! \binom{b\!-\!i\!-\!3}{i\!-\!t\!-\!1} \binom{a\!-\!i\!-\!2}{t} \binom{i}{t} \bigg]. \nonumber
  \end{gather}

  We observe that in \eqref{Eq.homology1} we can replace the upper summation index $a-i-2$ by $i$. In fact, if $i>a-i-2$, then for every $t \geq a-i-1$
  the $t$\textsuperscript{th} summand is zero since $\binom{a-i-3}{t-1}=\binom{a-i-2}{t}=0$. If $i<a-i-2$, we show that the $t$\textsuperscript{th}
  summand is zero for every $t \geq i+1$. In fact, $\binom{i}{t}=0$ and if $t>i+1$, also $\binom{i+1}{t}$=0. We only need to show that, if $t=i+1$,
  the first summand in \eqref{Eq.homology1} is zero. Notice that, $\binom{b-i-3}{-1} \neq 0$ if and only if $b-i-3=-1$. This means that $i=b-2$.
  On the other hand, $2i<a-2$, hence $2b-4<a-2$, that is $2b-2<a \leq b$. Thus $b<2$, in contradiction with $2 \leq a<b$.
  Therefore, $b-i-3 \neq -1$ and $\binom{b-i-3}{-1}=0$. Hence
  \begin{gather*}
    \mathrm{rank}\ \widetilde H_i (\Delta (  P(a,b) ); \mathbb Z ) = F_{(1,1)} \!+\! F_{(1,0)} \!+\! F_{(0,1)} \\
    = \sum_{t=0}^i \bigg[ \binom{b\!-\!i\!-\!3}{i\!-\!t} \binom{a\!-\!i\!-\!3}{t\!-\!1} \binom{i\!+\!1}{t} \!+\! \binom{b\!-\!i\!-\!2}{i\!-\!t}
    \binom{a\!-\!i\!-\!3}{t\!-\!1} \binom{i}{t} \!+\! \binom{b\!-\!i\!-\!3}{i\!-\!t\!-\!1} \binom{a\!-\!i\!-\!2}{t} \binom{i}{t} \bigg] 
  \end{gather*}
  \begin{gather*}
    = \sum_{t=0}^i \bigg[ \binom{a\!-\!i\!-\!3}{t\!-\!1} \bigg[ \binom{b\!-\!i\!-\!3}{i\!-\!t} \binom{i\!+\!1}{t} \!+\!
    \binom{b\!-\!i\!-\!2}{i\!-\!t} \binom{i}{t} \bigg] \!+\! \binom{b\!-\!i\!-\!3}{i\!-\!t\!-\!1} \binom{a\!-\!i\!-\!2}{t} \binom{i}{t} \bigg] \\
    = \sum_{t=0}^i \bigg[ \binom{a\!-\!i\!-\!3}{t\!-\!1} \bigg[ \binom{b\!-\!i\!-\!3}{i\!-\!t} \binom{i}{t} \!+\!
    \binom{b\!-\!i\!-\!3}{i\!-\!t} \binom{i}{t\!-\!1} \!+\! \binom{b\!-\!i\!-\!3}{i\!-\!t} \binom{i}{t} \!+\! \binom{b\!-\!i\!-\!3}{i\!-\!t\!-\!1} \binom{i}{t} \bigg] \\
    + \binom{a\!-\!i\!-\!3}{t} \binom{b\!-\!i\!-\!3}{i\!-\!t\!-\!1} \binom{i}{t} \!+\! \binom{a\!-\!i\!-\!3}{t-1} \binom{b\!-\!i\!-\!3}{i\!-\!t\!-\!1} \binom{i}{t} \bigg]\\
    = \sum_{t=0}^i \bigg[ 2 \binom{a\!-\!i\!-\!3}{t\!-\!1} \bigg[ \binom{b-i-3}{i-t} \binom{i}{t} \!+\! \binom{b-i-3}{i-t-1} \binom{i}{t} \bigg] \\
    + \bigg[ \binom{a\!-\!i\!-\!3}{t\!-\!1} \binom{b-i-3}{i-t} \binom{i}{t-1} \!+\! \binom{a\!-\!i\!-\!3}{t} \binom{b-i-3}{i-t-1} \binom{i}{t} \bigg] \bigg] \\
    = \sum_{t=0}^i \bigg[ 2 \binom{a\!-\!i\!-\!3}{t\!-\!1} \binom{b\!-\!i\!-\!2}{i\!-\!t} \binom{i}{t} \!+\!
   \binom{a\!-\!i\!-\!3}{t\!-\!1} \binom{b\!-\!i\!-\!3}{i\!-\!t} \binom{i}{t\!-\!1} \!+\! \binom{a\!-\!i\!-\!3}{t} \binom{b\!-\!i\!-\!3}{i\!-\!t\!-\!1} \binom{i}{t} \bigg].
  \end{gather*}

  Consider the three summands in the brackets above. The second summand is zero if $t=0$. The third one is zero if $t=i$. This follows since, by convention, $\binom{b-i-3}{-1}=1$
  if and only if $b-i-3=-1$. In this case, $b-2=i$ and, on the other hand, $i \leq a-2$, hence $a=b$, $t=i=a-2$ and $\binom{a-a+2-3}{a-2}=0$ for $a \geq 2$. Thus, summing over the
  second and third summand yields
  \begin{gather*}
    \sum_{t=1}^i \binom{a\!-\!i\!-\!3}{t\!-\!1} \binom{b\!-\!i\!-\!3}{i\!-\!t} \binom{i}{t\!-\!1} \!+\!
    \sum_{t=0}^{i-1} \binom{a\!-\!i\!-\!3}{t} \binom{b\!-\!i\!-\!3}{i\!-\!t\!-\!1} \binom{i}{t} \\
    = 2 \sum_{t=1}^i \binom{a\!-\!i\!-\!3}{t\!-\!1} \binom{b\!-\!i\!-\!3}{i\!-\!t} \binom{i}{t\!-\!1} =
    2 \sum_{t=0}^i \binom{a\!-\!i\!-\!3}{t\!-\!1} \binom{b\!-\!i\!-\!3}{i\!-\!t} \binom{i}{t\!-\!1}.
  \end{gather*}
  Substituting in the above expression, we get the desired formula.
\end{proof}

\begin{remark}
  Notice that, for $4 \leq a \leq b$, by Theorem \ref{T.homology} it follows that $\rank \widetilde H_1 (\Delta (  P(a,b) ); \mathbb Z ) = 4$.
\end{remark}

%

\begin{corollary} \label{C.LastNonzeroHomology}
  For every $2 \leq a \leq b$, $\mathrm{rank}\ \widetilde H_{a-2}(\Delta (  P(a,b) ); \mathbb Z ) = 2\binom{b-a}{a-2}$. In particular, if $b \geq 2a-2$, then $t_{(a,b)}=a-2$.
\end{corollary}

\begin{proof}
  If $i=a-2$, the expression \eqref{Eq.homology} has only the summand for $t=0$, hence
  $\mathrm{rank}\ \widetilde H_{a-2}(\Delta (  P(a,b) ); \mathbb Z ) = 2\binom{b-a}{a-2}$.
  In particular, if $2 \leq a \leq b < 2a-2$, then $\widetilde H_{a-2}(\Delta (  P(a,b) ); \mathbb Z ) = 0$. The second part of the claim follows from Proposition \ref{P.vanishing}.
\end{proof}

Another noteworthy property of the posets of proper divisibility is that the non-reduced simplicial homology is non-zero for \textit{every} degree between $0$ and $t_{(a,b)}$.

\begin{proof}[Proof of Proposition \ref{P.nojumps}]
  Let $2 \leq a \leq b$. Notice that $\rank H_0(\Delta(P(a,b));\ZZ) \neq 0$ since $\Delta(P(a,b)) \neq \emptyset$. Moreover, the claim is true for $a=2,3$ by Example \ref{E.a=2,3}.

  Hence, by \cite[Thm. 5.9]{BW96}, it suffices to show a falling chain of length $i+2$ in $P(a,b)^\ast$, for every $1 \leq i \leq t_{(a,b)}$. First assume $i=1$. By Lemma \ref{Falling chains}, the chain $\widehat{0}\!=\!(a,b) \rightarrow (a\!-\!1,2) \rightarrow (1,0) \rightarrow (0,0)\!=\!\widehat{1}$ has length $3$ in $P(a,b)^\ast$ and is falling. Now let $i \geq 2$.

If $b \geq 2a-2$, then $t_{(a,b)}=a-2$ by Corollary \ref{C.LastNonzeroHomology}. The chain
\begin{gather*}
m: \widehat{0}\!=\!(a,b) \rightarrow (a\!-\!1,b\!-\!2) \rightarrow (a\!-\!2,b\!-\!4) \rightarrow\! \cdots\! \rightarrow (a\!+\!1\!-\!i,b\!-\!2i\!+\!2) \rightarrow (a\!-\!i,2) \rightarrow (1,0) \rightarrow (0,0)\!=\!\widehat{1}
\end{gather*}
has $i+3$ elements, then its length is $i+2$. Since $i \geq 2$, we obtain $a+1-i<a$ and $b-2i+2<b$. Notice that $a-i \geq 2$ by assumption, then $b-2i+2 \geq 2a-2-2i+2 \geq 4$. By Lemma \ref{Falling chains}, $m$ is falling.

Let $4 \leq a \leq b<2a-2$. We consider the cases $a=4$ and $a=5$ separately. If $a=4$, then $b \in \{4,5\}$; if $a=5$, then $b \in \{5,6,7\}$. By Proposition \ref{Bounds on homology}, $t_{(4,4)}=t_{(4,5)}=1$ and $t_{(5,5)}=t_{(5,6)}=t_{(5,7)}=2$. For $i=1$, a falling chain of length $3$ is given above and for $i=t_{(a,b)}$, a falling chain of length $t_{(a,b)}$ is provided in the proof of Proposition \ref{Bounds on homology}.

Finally we assume $a \geq 6$. Notice that $a-2<b-1$. It suffices to show that there exists a falling chain
\begin{gather*}
c:(a-2,b-1) \rightarrow x_1 \rightarrow \cdots \rightarrow x_{i+1} = (0,0)
\end{gather*}
of length $i+1$ in $P(a-2,b-1)^\ast$, for every $2 \leq i \leq t_{(a,b)}$. In fact, by Lemma \ref{Falling chains}, it follows that the chain
\begin{gather*}
m:\widehat{0} = (a,b) \rightarrow (a-2,b-1) \rightarrow x_1 \rightarrow \cdots \rightarrow x_{i+1} = (0,0) = \widehat{1}
\end{gather*}
obtained by adding the element $(a,b)$ to the chain $c$, is a falling chain of length $i+2$ in $P(a,b)^\ast$.

Since $b \leq 2a-3$, there exists $k \geq 1$ such that $2a-3k-2 \leq b \leq 2a-3k$, then $t_{(a,b)}=a-2-k$, by Proposition \ref{Bounds on homology}. Hence $2(a-2)-3(k-1)-2 \leq b-1 \leq 2(a-2)-3(k-1)$. If $k=1$, then $b-1 \geq 2(a-2)-2$, thus $t_{(a-2,b-1)}=a-4$, by Proposition \ref{C.LastNonzeroHomology}. On the other hand, if $k>1$, then $t_{(a-2,b-1)}=a-2-2-(k-1)=a-3-k$, by Proposition \ref{Bounds on homology}. Hence $t_{(a-2,b-1)}=a-3-k$, for every $k \geq 1$. By induction on $a \geq 4$, for every $1 \leq h \leq a-3-k$, there exists a falling chain of length $h+2$ in $P(a-2,b-1)^\ast$. Since $1 \leq i-1 \leq a-3-k$ by assumption, we conclude the proof.
\end{proof}

\begin{remark}
Even though Theorem \ref{T.homology} provides an explicit formula for the rank of all homology groups, we do not see how to get Propositions \ref{P.nojumps}, \ref{P.vanishing} and \ref{Bounds on homology} as direct consequences Theorem \ref{T.homology}.
\end{remark}

\section{Euler characteristic of the order complex of $ P(a,b)$} \label{S.EulerChar}

In this section we give the proof of Theorem \ref{T.EulerChar}. We will mainly use
generating functions techniques.

\begin{proof}[Proof of Theorem \ref{T.EulerChar}]
  First we set $f(u,v)$ as the generating series

  \[
    f(u,v) = \sum_{a=2}^{\infty} \sum_{b=a}^{\infty} \widetilde{\chi}(\Delta(P(a,b))) u^a v^b
  \]
  of the reduced Euler-characteristic of $\Delta(P(a,b))$
  \[
    \widetilde{\chi}(\Delta(P(a,b)))  =  \sum_{i \geq 0} (-1)^i \, \rank \widetilde{H}_i(\Delta(P(a,b));\ZZ).
  \]
  By Theorem \ref{T.homology}, we have that, for $2 \leq a \leq b$,
  \[
    \widetilde{\chi}(\Delta(P(a,b)))  =  \sum_{i=0}^{\infty} (-1)^i 2
                 \sum_{t=0}^i \binom{a-3-i}{t-1} \bigg[
                    \binom{i}{t}\binom{b-2-i}{i-t} +
                    \binom{i}{t-1}\binom{b-3-i}{i-t} \bigg].
  \]
  The expression on the right-hand side makes sense for all $a, b \geq 0$ therefore we can
  formally write the series
  \begin{eqnarray} \label{Eq.Ansatz}
    \bar{f}(u,v) & = & \sum_{a=0}^{\infty} \sum_{b=0}^{\infty}
                 \sum_{i=0}^{\infty} (-1)^i 2
                 \sum_{t=0}^i \binom{a-3-i}{t-1} \bigg[
                    \binom{i}{t}\binom{b-2-i}{i-t} +
                    \binom{i}{t-1}\binom{b-3-i}{i-t} \bigg] u^av^b.
  \end{eqnarray}


  Interchanging the summation in \eqref{Eq.Ansatz} and using elementary generating function identities
  (see, e.g., identity {\bf d.} in \cite[p. 209]{S99}) we obtain:
  \begin{gather*}
    \bar f(u,v) = 2 \sum_{i=0}^{\infty} (-1)^i
       \left[ \frac{u^{i+2}v^{2i+2}}{(1-v)^{i+1}}
          \sum_{t=0}^i \binom{i}{t}
            \left( \frac{u(1-v)}{(1-u)v} \right)^{\!t} +
                   \frac{u^{i+2}v^{2i+3}}{(1-v)^{i+1}}
                      \sum_{t=0}^i \binom{i}{t-1}
                         \left( \frac{u(1-v)}{(1-u)v} \right)^{\!t} \right]
  \end{gather*}
  \begin{gather*}
    = 2 \sum_{i=0}^{\infty} (-1)^i
          \left[ \frac{u^{i+2}v^{2i+2}}{(1-v)^{i+1}}
             \left(1 + \frac{u(1-v)}{(1-u)v} \right)^{\!i} +
                   \frac{u^{i+2}v^{2i+3}}{(1-v)^{i+1}}
                   \frac{u(1-v)}{(1-u)v} \left( \!\!
                     \left(1 + \frac{u(1-v)}{(1-u)v} \right)^{\!i} -
                     \left(\frac{u(1-v)}{(1-u)v} \right)^{\!i} \right) \right]\\
   = 2 \sum_{i=0}^{\infty} (-1)^i
          \left[ \frac{u^{i+2}v^{i+2}(u+v-2uv)^i}{(1-u)^i(1-v)^{i+1}} +
                 \frac{u^{i+3}v^{i+2}(u+v-2uv)^i}{(1-u)^{i+1}(1-v)^{i+1}} -
                 \frac{u^{2i+3}v^{i+2}}{(1-u)^{i+1}} \right] \\
   = 2 \sum_{i=0}^{\infty} (-1)^i
     \left[ \frac{u^{i+2}v^{i+2}(u+v-2uv)^i(1-uv)}{(1-u)^{i+1}(1-v)^{i+1}} -
            \frac{u^{2i+3}v^{i+2}}{(1-u)^{i+1}} \right]
   = 2 \left( \frac{u^2v^2}{2uv-u-v+1} - \frac{u^3v^2}{u^2v-u+1} \right).
  \end{gather*}

  Notice that, since
  \[
    \frac{1}{u^2v-u+1} = \frac{1}{1-u(1-uv)} =
    \sum_{n=0}^{\infty} (u^n(1-uv)^n),
  \]
  in the Taylor expansion of
  $\frac{u^3v^2}{u^2v-u+1}$ at $(0,0)$ only monomials $u^av^b$,
  where $a>b$, appear with non-zero coefficient.
  Hence, for every
  $2 \leq a \leq b$, the coefficient of $u^av^b$ in the Taylor expansion of
  $\frac{2u^2v^2}{2uv-u-v+1}$ at $(0,0)$ is $\widetilde \chi ( \Delta (  P(a,b) ) )$,
  which in turn is also the coefficient of $u^av^b$ in $f(u,v)$.

  On the other hand, let us compute the generating function of the right-hand
  side of \eqref{Eq.EulerChar}. Using a similar approach as above, we consider the series
  \[
    \bar{g}(u,v) = \sum_{a=0}^{\infty}
               \sum_{b=0}^{\infty} (-1)^a \cdot 2 \
                 \left[ \sum_{h=0}^{\lfloor \frac{a}{2} \rfloor -1} (-1)^h
                    \binom{a-2}{h} \binom{b-a}{a-2-2h} \right] u^av^b.
  \]

  Notice that, if $h \geq \lfloor \frac{a}{2} \rfloor$, then
  $\binom{b-a}{a-2-2h}=0$. Hence, we may extend the sum on $h$ up to
  infinity. Interchanging the order of summation and using
  the identity {\bf d.} in \cite[p. 209]{S99}, we have
  \begin{gather*}
    \bar{g}(u,v) = 2 \sum_{a=0}^{\infty} (-1)^a
               \sum_{h=0}^{\infty} (-1)^h \binom{a-2}{h}
                 \left[ \sum_{b=0}^{\infty}
                   \binom{b-a}{a-2-2h} v^b \right] u^a \\
           = 2 \sum_{a=0}^{\infty} (-1)^a \sum_{h=0}^{\infty} (-1)^h
                 \binom{a-2}{h} \frac{v^{2a-2-2h}u^a}{(1-v)^{a-1-2h}}
           = 2 \sum_{a=0}^{\infty} (-1)^a \frac{v^{2a-2}u^a}{(1-v)^{a-1}}
                 \sum_{h=0}^{\infty} (-1)^h \binom{a-2}{h}
                   \left( \frac{1-v}{v} \right)^{\!2h} \\
           = 2 \sum_{a=0}^{\infty} (-1)^a \frac{v^{2a-2}u^a}{(1-v)^{a-1}}
                 \left( 1-\frac{(1-v)^2}{v^2} \right)^{\!a-2}
           = \frac{2(1-v)v^2}{(2v-1)^2} \sum_{a=0}^{\infty} (-1)^a
                \left( \frac{u(2v-1)}{(1-v)} \right)^{\!a} \\
           = \frac{2(1-v)v^2}{(2v-1)^2} \frac{u^2(2v-1)^2}{(1-v)(2uv-u-v+1)}
           = \frac{2u^2v^2}{2uv-u-v+1}.
  \end{gather*}

  Since we already know that, for $2 \leq a \leq b$, the coefficient of $u^av^b$ in
  the Taylor expansion of $\frac{2u^2v^2}{2uv-u-v+1}$ is
  $\widetilde \chi ( \Delta (  P(a,b) ) )$, the assertion follows.
\end{proof}

\begin{corollary} \label{C.ZeroEulerChar}
  For every $2 \leq a \leq b$, $\widetilde \chi ( \Delta (  P(a,b) ) ) = 0$ if $a=b$ is odd. Moreover, if $a=b$ is even,
  \begin{equation} \label{Eq.EulerCharAeven}
\widetilde \chi ( \Delta (  P(a,b) ) ) = (-1)^{\frac{a-2}{2}} \cdot 2 \binom{a-2}{\frac{a-2}{2}}.
\end{equation}
\end{corollary}

\begin{proof}
Notice that, using the formula \eqref{Eq.EulerChar}, when $a=b$, the second binomial coefficient is $\binom{0}{a-2-2h}$. This is zero if $a$ is odd, hence $\widetilde \chi ( \Delta (  P(a,a) ) ) = 0$, and it is $1$ when $a$ is even and $h=\frac{a-2}{2}$. In the last case, we get the formula \eqref{Eq.EulerCharAeven}.
\end{proof}

\section*{Acknowledgement}

We thank Michelle Wachs for pointing us to the example by Jay Schweig from \cite{SW08} and Jay Schweig for interesting comments.

\end{document}